\documentclass[11pt]{article}
\usepackage[utf8]{inputenc}
\usepackage[english]{babel}
\usepackage{amsmath}
\usepackage{amssymb}
\usepackage{float}

\author{Brian Alspach and Aditya Joshi\\
School of Information and Physical Sciences\\
University of Newcastle\\
Callaghan, NSW 2308\\
Australia\\
brian.alspach@newcastle.edu.au\\
aditya.joshi@uon.edu.au\\ 
 }

\date{\today}
\title{On the k-spanning cyclability of 4-valent Cayley graphs on Abelian groups}

\usepackage{graphicx}
\graphicspath{ {./images/} }
\usepackage{tikz}

\usepackage{mathtools}

\usepackage{amsthm}

\newtheorem{theorem}{Theorem}[section]
\newtheorem{corollary}{Corollary}[theorem]
\newtheorem{lemma}[theorem]{Lemma}

\newtheorem{definition}[theorem]{Definition}

\providecommand{\keywords}[1]
{
  \small	
  \textbf{Keywords:} #1
}

\providecommand{\msc}[1]
{
  \small	
  \textbf{MSC 2020:} #1
}

\begin{document}
\maketitle

\abstract{A graph $X$ is $k$-spanning cyclable if for any subset $S$ of $k$ distinct vertices there is a 2-factor of $X$ consisting of $k$ cycles such that each vertex in $S$ belongs to a distinct cycle. In this paper we examine the $k$-spanning cyclability of 4-valent Cayley graphs on Abelian groups.} \newline
\msc{05C70} \newline

\section{Introduction}
Graphs consist of vertices and edges joining pairs of distinct vertices such that neither loops nor multiple edges are allowed. If $X$ is a graph, its vertex set is denoted $V(X)$ and its edge set is denoted $E(X)$. The \textit{order} of a graph $X$ is $|V(X)|$ and the size is $|E(X)|$. The number of edges incident with a given vertex $v$ is its \textit{valency} and is denoted $val(v)$. All graphs throughout this paper are connected and  finite. 

An edge joining vertices $u$ and $v$ is denoted $[u,v]$. Continuing in this manner, a path of length $t$ from $u_0$ to $u_t$ is a connected subgraph of order $t+1$ all of whose vertices have valency 2 other than $u_0$ and $u_t$ which have valency 1. The path is denoted $[u_0,u_1,...,u_t]$, where $[u_i,u_{i+1}]$ is an edge for $0 \leq i \leq t-1$. If we start with a path of length $t$ from $u_0$ to $u_t$ and add the edge $[u_0,u_t]$, we obtain a cycle of length $t+1$ and use the notation $[u_0,u_1,...,u_t,u_0]$. 

A 2-factor in a graph $X$ is a spanning subgraph of $X$ such that every vertex has valency $2$. A 2-factor $F$ of $X$ \textit{separates} a set $A$ of $k$ vertices in $V(X)$ if $F$ is composed of $k$ cycles and $A$ intersects the vertex set of each cycle in $F$ in a single vertex. 

\begin{definition}\label{D1.1}
\normalfont A graph $X$ is   \textit{$k$-spanning cyclable}  if for every $A\subseteq V(X)$ such that $|A| = k$ there is a 2-factor of $X$ separating $A$. 
\end{definition}

The preceding concept has been studied because of the general problem of embedding cycles in graphs. Lin, Tan, Hsu and Kung \cite{L1} considered $k$-spanning cyclability for $n$-cubes. Yang, Hsu, Hung and Cheng \cite{Y1} considered 2-spanning cyclability for generalized Petersen graphs. In a recent paper, Qiao, Sabir and Meng studied $k$-spanning cyclability of Cayley graphs on symmetric groups whose connection sets are transposition trees \cite{Q1}. The authors of this paper have examined the $k$-spanning cyclability of honeycomb toroidal graphs \cite{A1}. 

We now consider the $k$-spanning cyclability of 4-valent Cayley graphs on Abelian groups. It is easy to see that a regular graph of valency 4 cannot be $k$-spanning cyclable for $k\geq4$ since a cycle passing through any vertex $v$ must contain two of its adjacent vertices. Hence, 4-valent graphs are at most 3-spanning cyclable. 
Of course, a graph is Hamiltonian if and only if it is 1-spanning cyclable. It is also well known that all Cayley graphs on a finite Abelian group of order at least 3 are Hamiltonian, and hence all 4-valent Cayley graphs on Abelian groups are 1-spanning cyclable. This leaves us with the 2-spanning and 3-spanning cyclability cases. 

\section{The 3-valent case}
For a preliminary understanding of the general approach in examining the 4-valent graphs, and for more completness, we look at the 2-spanning cyclability of 3-valent Cayley graphs on Abelian groups. 

\begin{theorem}
If $X$ is a connected 3-valent Cayley graph on an Abelian group, then $X$ is 2-spanning cyclable if and only if it is isomorphic to $Q_3$ or $K_2\Box C_n$, where $n\geq 4$ 
\end{theorem}

\begin{proof}
We consider two cases of the connection set $S$. The first is when $S$ contains three involutions and the other is when it contains exactly one. Let the connection set be $S = \{ a,b,c\}$ where all three elements are involutions. In this case there are two possible graphs. The first is when one of the involutions in $S$ is equal to the sum of the other two involutions. In this case the Cayley graph is isomorphic to $K_4$. However, in order to separate two vertices we must have two cycles each of length at least 3, and hence at least 6 vertices in the graph. Hence $K_4$ is not 2-spanning cyclable. When none of the involutions in $S$ are the sum of the other two we have the the 3-dimensional cube $Q_3$. It is also easy to see that it is 2-spanning cyclable. 

We now consider the second case when
$S=\{a,\pm s\}$ contains a single involution $a$.  By considering the subgraph generated by $\pm s$, it is easy to see that $X$ is isomorphic to either the circulant graph (defined in section 3) with connection set $\{\pm 1,n/2\}$, $n$ even, or the graph $K_2\Box C_n$, where $n=\mathrm{ord}(s)$.

First consider the circulant graph of even order $n$ with connection set $\{\pm 1,n/2\}$.  If there is a 2-factor with two cycles separating 0 and $n/2$, then one cycle must contain the path $[n-1,0,1]$ and the other cycle must contain the path $[n/2-1,n/2,n/2+1]$.  We see immediately that the 2-factor can contain no diameter edges.  Hence, there is no 2-factor separating 0 and $n/2$.

Now consider $K_2\Box C_n$. This graph has order 6 when $n=3$ and the only 3-cycles in the graph are the two vertex-disjoint 3-cycles $T_1$ and $T_2$ forming the product. Thus, there is no 2-factor separating two vertices of $T_1$ because there are only six vertices available. On the other hand, when $n\geq 4$, it is trivial to establish that $K_2\Box C_n$ is 2-spanning cyclable.  

\end{proof}

As reflected in the above proof the approach taken to examine the 4-valent Cayley graphs on Abelian groups is to first look at the possible structures of the graph, determined by different cases of the connection set, and then to examine the $k$-spanning cyclability of each of these structures. We begin by looking at the possible graph structures in the next section. 

\section{Graph structure}
In this section we describe the possible structures of the Cayley graphs in question. We will then examine the spanning cyclability of these structures in the subsequent sections.
  
We first define two classes of graphs which will help describe the structures of the Cayley graphs in question. Let $P_m$ denote the path of order $m$ and length $m-1$ and let $C_n$ denote the cycle of order $n$.  The {\it pseudo-Cartesian product} of 
$C_m$ and $C_n$, $m,n\geq 3$, with {\it jump $\ell$} is obtained by starting with the Cartesian
product $P_m \Box C_n$ and adding the edges from $u_{m-1,j}$ to $u_{0,j+\ell}$, where the second 
coordinates are computed modulo $n$.  The latter edges are said to have jump $\ell$.  The notation for
this product is $C_m \Box_\ell  C_n$. Of course, when $\ell=0$, the pseudo-Cartesian product is just the
ordinary Cartesian product. The second class of graphs are called \textit{circulant} graphs and are defined to be Cayley graphs on a cyclic group. 

There are four cases for the connection set $S$ of a 4-valent Cayley graph on an Abelian group $G$. The first is when $S$ contains four involutions. The second is when $S$ consists of two involutions and one element of order greater than 2 and its inverse. The third case is when $S$ has no involutions and contains at least one element whose order is greater than 2 but less than $|G|$. Finally, the fourth case is when $S$ contains two elements of order $|G|$ and their inverses. We now examine the structure of the Cayley graphs in each of these cases. 

When $S$ contains four involutions, there are only two types of Cayley graphs which can be constructed. It is easy to verify that all such graphs are 2-spanning cyclable. The first type is a Cayley graph of order 8 but since a 3-spanning cyclable graph requires there to be three disjoint cycles with each cycle containing at least three vertices, the graph must necessarily contain at least nine vertices and so a Cayley graph of order 8 is not 3-spanning cyclable. The other case is when the Cayley graph is isomorphic to the 4-dimensional cube $Q_4$ and it is easily verifiable that the graph is 3-spanning cyclable \cite{L1}.

For the second case the connection set is $S =\{a,b,\pm s \}$, where $a$ and $b$ are involutions and $s$ has order $r > 2$, for which there are a couple of possibilities based on the fact a cyclic group has no elements of order 2 when it has odd order and a unique element of order 2 when it has even order. Let $H$ be the cyclic subgroup of order $r$ generated by $s$.  If $H$ has odd order, then the graph is $C_4\Box C_r$.  Similarly, if none of $a$, $b$ or $a+b$ belongs to $H$, then the graph is again $C_4\Box C_r$. If $a+b$ belongs to $H$, then the graph is
isomorphic to the pseudo-Cartesian product $C_r\Box_2 C_4$. Finally, if either $a$ or $b$ belong to $H$, then the graph is not a pseudo-
Cartesian product and is isomorphic to $Y\Box K_2$, where $Y$ is the circulant
graph of order $r$ with connection set $\{\pm 1,r/2\}$. This graph needs to be
checked separately.

The third case is similar to the second. Assume that the connection set is $S = \{\pm a, \pm b: a \neq -a, b \neq -b \}$. Without loss of generality we may assume that $ord(a) <|G|$. The orbit of $e$ under the cyclic group generated by $a$ forms an $ord(a)$-cycle. Similar to the reasoning of the second case we can use $b$ to connect the vertices in this cycle to other cycles of length $ord(a)$. We can see that the resulting graph is a pseudo-Cartesian product of the cycles $C_{ord(a)}$ and $C_m$, where $m \geq 3$. 

Finally, for the final case we have $S = \{\pm a, \pm b: ord(a) = ord(b) = |G|\}$. We can see that every element of $G$ can be represented as the power of the element $a$, and hence $G$ is a cyclic group. Since $b$ also generates $G$, and is not equal to $a$, we have that $b = ka$ where $gcd(k, |G|) =1$. Such a graph is a circulant graph and is isomorphic to the Cayley graph on the additive group $\mathbb{Z}/n\mathbb{Z}$ with connection set $S = \{\pm 1, \pm s\}$, which is denoted by $\mathrm{circ(n; \pm 1, \pm s)}$, where $gcd(s,n)=1$.

In summary the three classes of graphs that remain to be examined are $C_m \Box_\ell C_n$, where $m,n \geq 3$, $Y\Box K_2$, where $Y$ is the circulant graph of even order $n$ with connection set $\{\pm 1,n/2\}$, and $\mathrm{circ(n; \pm 1, \pm s)}$, where $gcd(n,s)=1$. These graphs are examined in the following three subsequent sections.

\section{The product of cycles case}

We first note that there are three cases for the 3-spanning cyclability of the pseudo-Cartesian product of two cycles. The first is when all three vertices are in distinct columns, the second is when two vertices are in the same column and the other is in some other column, and finally the third case is when all three vertices are in the same column. Using the automorphism which maps a vertex to the next column we may assume that one of the three vertices belongs to the first column in the first case. We may also assume that the two vertices that are in the same column belong to the first column for the second case. Finally, all three vertices can be in the first column for the third case. The 2-spanning cyclability is similar to and simpler than the above. 

We use a constructive technique throughout this section which we now describe in detail for
one example. We then leave it to the reader to apply the technique in all other situations.
The basic idea is the following. We start with a given 2-factor composed of two or three
cycles for small values of $m$ and $n$. We then insert a certain number of rows and/or columns
obtaining a 2-factor composed of the same number of cycles for other values of $m$ and $n$.

We use Figure 1 for the description. The figure shows a 2-factor in the graph $P_3\Box C_4$.
Note that this 2-factor separates both $u_{0,0}$ and $u_{1,0}$ from any vertex of the form
$u_{i,1}$, that is, any vertex in the 1-row. If you now subdivide each horizontal edge from
the 0-column to the 1-column $r$ times, we obtain a 2-factor with two cycles in $P_{r+3}\Box C_4$ separating every vertex of the 0-row, except $u_{r+2,0}$, from every vertex of the 1-row.  We
refer to this as the ability to insert an arbitrary number of columns.

Note that all three columns have an edge from the 0-row to the 3-row. These edges may be be
subdivided to insert an arbitrary number of rows without increasing the number of cycles in
the resulting 2-factor. This is what we refer to as inserting an arbitrary number of rows.

There are occasions when we cannot insert an arbitrary number of rows or columns, but we may
always insert any even number of rows or columns.  For example, consider the edge $[u_{0,1},u_{0,2}]$.  Subdivide this edge into the 3-path $[u_{0,1},u_{0,a},u_{0,b},u_{0,2}]$. Then
replace the edge $[u_{0,a},u_{0,b}]$ with the path $[u_{0,a},u_{1,a},u_{2,a},u_{2,b},
u_{1,b},u_{0,b}],$ where the obvious new vertices have been created. Then relabel the
rows with subscripts from 0 to 5. It is clear that we may repeat this an arbitrary number of
times, and that we may do it to produce an even number of columns.

We conclude with a particular example starting with Figure.  Suppose we want a 2-factor
in $P_9\Box C_7$ separating $u_{0,0}$ and $u_{0,4}$.  Start with Figure 1 and subdivide each edge from the 0-row to the 3-row once.  Let the second subscript of each of these new vertices be 4. Now insert two new rows between $u_{0,1}$ and $u_{0,2}$. The resulting 2-factor
separates $u_{0,0}$ and $u_{0,4}$ as required.

\begin{figure}[H]
\centering

\begin{picture}(200,120)(-80,-50)

\multiput(0,0)(20,0){3}{\multiput(0,0)(0,20){4}{\circle*{5}}}

\multiput(0,20)(0,100){1}{\line(0,1){20}}
\multiput(40,0)(0,100){1}{\line(0,1){20}}
\multiput(40,40)(0,100){1}{\line(0,1){20}}

\multiput(0,0)(0,100){1}{\line(1,0){20}}
\multiput(0,20)(0,100){1}{\line(1,0){40}}
\multiput(0,40)(0,100){1}{\line(1,0){40}}
\multiput(0,60)(0,100){1}{\line(1,0){20}}

\qbezier(0,0)(-15,30)(0,60)
\qbezier(20,0)(5,30)(20,60)
\qbezier(40,0)(55,30)(40,60)

\put(0,-35){{\sc Figure 1.} }
\end{picture} \newline

\end{figure}

\begin{lemma}\label{L3.1} 
Two vertices in the same column of $P_m \Box C_n$, $m\geq3$ and $n\geq4$ even,  have a 2-factor separating them.
\end{lemma}

\begin{proof}
Consider Figure 1. We may rotate the graph vertically so that the vertices are either in the top two rows or the top row and the row two below it. Then inserting columns allows us to have the two vertices in any of the columns $0,1,...,m-2$. 
It is straightforward to see that we can add any even number of rows between any two rows and extend the 2-factor appropriately. Doing so will give a 2-factor separating any two vertices in all but the last column. To separate vertices in the last column we simply flip our extended 2-factor front to back. This completes the proof.
\end{proof}

\begin{figure}[H]
\centering

\begin{picture}(200,100)(-80,-20)

\multiput(-50,0)(20,0){3}{\multiput(0,0)(0,20){5}{\circle*{5}}}

\multiput(-50,0)(-50,100){1}{\line(0,1){20}}
\multiput(-50,40)(-50,100){1}{\line(0,1){20}}

\multiput(-30,0)(-50,100){1}{\line(0,1){20}}

\multiput(-10,0)(-50,100){1}{\line(0,1){40}}

\multiput(-10,60)(-50,100){1}{\line(0,1){20}}

\multiput(-50,20)(-50,100){1}{\line(1,0){20}}
\multiput(-50,40)(-50,100){1}{\line(1,0){40}}
\multiput(-50,60)(-50,100){1}{\line(1,0){40}}
\multiput(-50,80)(-50,100){1}{\line(1,0){20}}

\qbezier(-50,0)(-65,40)(-50,80)
\qbezier(-30,0)(-15,40)(-30,80)
\qbezier(-10,0)(5,40)(-10,80)

\multiput(50,0)(20,0){4}{\multiput(0,0)(0,20){5}{\circle*{5}}}

\multiput(50,0)(50,100){1}{\line(0,1){20}}
\multiput(50,40)(50,100){1}{\line(0,1){20}}

\multiput(70,0)(50,100){1}{\line(0,1){20}}

\multiput(90,0)(50,100){1}{\line(0,1){40}}

\multiput(90,60)(50,100){1}{\line(0,1){20}}

\multiput(50,20)(50,100){1}{\line(1,0){20}}
\multiput(50,40)(50,100){1}{\line(1,0){40}}
\multiput(50,60)(50,100){1}{\line(1,0){40}}
\multiput(50,80)(50,100){1}{\line(1,0){20}}

\multiput(90,0)(110,0){1}{\line(1,0){20}}
\multiput(90,80)(110,80){1}{\line(1,0){20}}
\multiput(110,0)(110,80){1}{\line(0,1){80}}

\qbezier(50,0)(35,40)(50,80)
\qbezier(70,0)(85,40)(70,80)

\put(0,-35){{\sc Figure 2.} }
\end{picture} \newline

\end{figure}

The next result extends Lemma \ref{L3.1} to $n$ odd.

\begin{lemma}\label{L3.2} 
Two vertices in the same column of $P_m \Box C_n$, $m\geq3$ and odd $n\geq5$, can be separated by a 2-factor. 
\end{lemma}

\begin{proof}
Consider Figure 2. Similar to the proof of Lemma \ref{L3.1} we may assume that the two vertices to separate are in the top two rows or the top row and the row two below it. We may insert an arbitrary number of even columns between the 0-column and the 1-column in both figures. This gives us 2-factors that separate vertices of the top row from vertices in either of the next two rows below for columns $0,1,2,\ldots,m-3$. We then may insert an arbitrary even number of rows between the top row and the next row below to separate a vertex from the top row and a vertex an arbitrary distance below it. (Note that we may restrict this distance to $n/2$.) Finally, we may insert an arbitrary even number of rows between the bottom two rows to achieve the desired value of $n$. The preceding works for columns $0,1,\ldots,m-3$.  To cover columns $m-2$ and $m-1$, we interchange left and right. This completes the proof.
\end{proof}

\begin{lemma}\label{L3.3} 
$P_m \Box C_n$ is Hamiltonian for all $m\geq 1$ and $n \geq 3$.
\end{lemma}
\begin{proof}
Easy to see
\end{proof}

\begin{theorem}\label{T3.4}
$P_m \Box C_n$ is 2-spanning cyclable for all $m \geq 3$ and $n\geq4$. 
\end{theorem}

\begin{proof}
Lemmas \ref{L3.1} and \ref{L3.2} cover the case when the two vertices are in the same column. If the two vertices are in columns $i$ and $j$, where $i <j$, then we do the following. Look at the subgraph induced on columns $0,1,...,i$ and the subgraph induced on columns $i+1,i+2,...,m-1$. The former is isomorphic to $P_{i+1} \Box C_n$ and the latter is isomorphic to $P_{m-i-1} \Box C_n$. Both of them are Hamiltonian by Lemma \ref{L3.3} yielding a 2-factor separating the two vertices.  
\end{proof}

\begin{corollary}\label{C3.4.1}
$C_m \Box_\ell C_n$ is 2-spanning cyclable for all $m\geq 3$ and $n\geq 3$. 
\end{corollary}

\begin{proof}
It is easy to see that $C_m \Box_\ell C_n$ is 2-spanning cyclable for $m=n=3$ and is left to the reader to verify. The remaining cases are covered by Theorem \ref{T3.4}.
\end{proof}

\begin{lemma}\label{L3.5} 
If $C_m \Box C_n$ has a 2-factor using no edges between column $m-2$ and column $m-1$ which separates three vertices $x,y$ and $z$ in $C_m \Box C_n$, then $C_m \Box_\ell C_n$ has a 2-factor separating $x,y$ and $z$ unless some of the vertices lie in column $m-1$. In the latter case the vertices in column $m-1$ are shifted by $-\ell$ to obtain the set of separated vertices. 
\end{lemma}

\begin{proof}
Let $F$ be the 2-factor of $C_m \Box C_n$ separating $x,y$ and $z$. Because there are no edges of $F$ joining vertices of column $m-2$ to vertices of column $m-1$, if we shift the vertices of column $m-1$ by $-\ell$, $F$ is transformed into a 2-factor $F^{'}$ in $C_m \Box_\ell C_n$. Any vertex of $x,y,z$ not lying in column $m-1$ does not change position and any vertex in column $m-1$ changes by $-\ell$. The conclusion follows. 
\end{proof}

\begin{lemma}\label{L3.6} The pseudo-Cartesian product $C_m \Box_\ell C_3$ is 3-spanning cyclable if and only if
$m\geq 4$ and $\ell=0$.
\end{lemma}
\begin{proof} We first consider $m=3$.  Because $|C_3 \Box_\ell C_3|=9$, the only possible 2-factors which
may be used to separate three vertices consist of three 3-cycles.  Choose the three vertices $u_{0,0},
u_{0,1}\mbox{ and }u_{1,0}$. If there are three 3-cycles separating the preceding vertices, then the 3-cycle containing $u_{0,0}$ must contain the edge $[u_{0,0},u_{0,2}]$ and a vertex from
the 2-column.  This is not possible as $\ell$ is unique. Thus, $C_3\Box_\ell C_3$ is not 3-spanning cyclable for all $\ell$.

The next step is to show that $C_m\Box C_3$ is 3-spanning cyclable for all $m\geq 4$.  It is clear that
if we choose three vertices lying in distinct rows or distinct columns, we easily may find a 2-factor
separating the three vertices.  Hence, we may assume two of the vertices are $u_{0,0}$ and $u_{0,1}$  
and the third vertex is $u_{i,0}$ or $u_{i,1}$ for some $i\neq 0$.  We assume that the third vertex
is $u_{i,0}$.  Use the cycle formed by the 1-row and it contains $u_{0,1}$.  To obtain a cycle containing
$u_{0,0}$, take the edges $[u_{0,0},u_{0,2}]$ and $[u_{i-1,0},u_{i-1,2}]$, and join them with the
respective paths along the 0-row and the 2-row from left to right.  Finally, to get the cycle containing
$u_{i,0}$ which completes a 2-factor, use the edges $[u_{i,0},u_{i,2}]$ and $[u_{m-1,0},u_{m-1,2}]$
and the respective paths along the 0-row and the 2-row from left to right.  

We may do the obvious analogous construction when the third vertex is $u_{i,1}$.  This shows that
$C_m\Box C_3$ is 3-spanning cyclable for $m\geq 4$.  To complete the proof we need to show that
neither $C_m\Box_1 C_3$ nor $C_m \Box_2 C_3$ are 3-spanning cyclable.

Choose the three vertices from the 0-column and consider $\ell=1$.  Because three cycles separating
$u_{0,0},u_{0,1}\mbox{ and }u_{0,2}$ may not contain any edge in the 0-column, they must contain
the respective 2-paths $$[u_{1,0},u_{0,0},u_{m-1,2}],[u_{1,1},u_{0,1},u_{m-1,0}]\mbox{ and }
[u_{1,2},u_{0,2},u_{m-1,1}].$$

The cycles can use no edge of the $(m-1)$-column which implies that the 2-path ending at $u_{m-1,j}$
must use the edge to $u_{m-2,j}$.  It is easy to see this must continue as we extend the paths from
right to left.  When we reach the 2-column, the paths may not be extended and the graph is not
3-spanning extendable.  Essentially the same argument works for $\ell=2$ and the proof is complete.
\end{proof}

\begin{lemma}\label{L3.7}
The pseudo-Cartesian product $C_3 \Box_\ell C_n$ is 3-spanning cyclable if and only if $n\geq5$ or, $n=4$ and $\ell \neq 2$. 
\end{lemma}

\begin{proof}
The case when $n=3$ is settled in Lemma \ref{L3.6} above. Note that by using Lemma \ref{L3.5} it suffices to show that there is a 2-factor which separates the three vertices and doesn't use the edges between the last two columns for when $\ell =0$.

We now show that the graphs are 3-spanning cyclable for $n\geq 6$.  We can assume that $\ell =0$ for this case.

The first case is when all three vertices are in different columns. In this case we simply take the three cycle columns to be our 2-factor. 

The second case is when two vertices are in the same column and the other is not. Without loss of generality we may assume that two of the vertices are $u_{0,0}$ and $u_{0,j}$, where $j\neq0$ and the other vertex is $u_{2,k}$. We can take the last column cycle to be the cycle containing $u_{2,k}$. The next cycle can be formed by starting at $u_{0,0}$ and moving one up or one down depending on where $u_{0,j}$ is, and then moving one to the right, then one vertex up or down back to the vertex in the same row as $u_{0,0}$ and then back to $u_{0,0}$. We can then make a cycle using the remaining vertices in a similar manner.

The final case is when all three vertices are in the same column. Again we may assume that all three vertices are in the first column. There are three subcases to consider. When no two vertices are adjacent, exactly two are adjacent, and when all three are adjacent to each other. 
For the first subcase we may assume that the three vertices are $u_{0,0}, u_{0,j}$ and $u_{0,k}$, where $1 < j < k-1 < n-2$. The 2-factor for this case is 
\begin{center}
$[u_{0,0}, u_{1,0},u_{1,n-1},...,u_{1,k+1},u_{0,k+1},...,u_{0,n-1},u_{0,0} ]$, 

$[u_{0,1},...,u_{0,k-2},u_{1,k-2},...,u_{1,1},u_{0,1}]$ and 

$[u_{0,k-1},u_{1,k-1},u_{1,k},u_{0,k},u_{2,k},u_{2,k+1},...,u_{2,k-1},u_{0,k-1}]$.
\end{center}

For the second subcase we may assume that the three vertices are $u_{0,0},u_{0,1}$ and $u_{0,k}$, where $2 < k < n-1$. If $k\geq4$, then we can use the same 2-factor as above. If $k=3$, we use the 2-factor 
\begin{center}
$[u_{0,0}, u_{1,0},u_{1,n-1},...,u_{1,5},u_{0,5},...,u_{0,n-1},u_{0,0} ]$, 

$[u_{0,1},u_{0,2},u_{1,2},u_{1,1},u_{0,1}]$ and 

$[u_{0,3},u_{1,3},u_{1,4},u_{0,4},u_{2,4},u_{2,5},...,u_{2,3},u_{0,3}]$.
\end{center}

Finally, for the last subcase we may assume the three vertices are $u_{0,0},u_{0,1}$ and $u_{0,2}$. The 2-factor for this will be 
\begin{center}
$[u_{0,0}, u_{1,0},u_{1,n-1},...,u_{1,5},u_{0,5},...,u_{0,n-1},u_{0,0} ]$, 

$[u_{0,1},u_{1,1},u_{1,2},u_{1,3},u_{1,4},u_{0,4},u_{2,4},u_{2,5},...,u_{2,1},u_{0,1}]$ and  

$[u_{0,2},u_{0,3},u_{2,3},u_{2,2},u_{0,2}]$.
\end{center}

We now verify the cases when $n=4$ and $n=5$.  Note that proofs for $n\geq 6$ remain valid whenever the three vertices are not in the same column. Hence, the case when three vertices are in the same column remains. We start with $n=5$. Without loss of generality the only subcases to check are when the vertices are $\{ u_{0,0},u_{0,1},u_{0,3}\}$ and $\{ u_{0,0},u_{0,1},u_{0,2}\}$.  The 2-factors for the first subcase, for $\ell = 0,...,4$ are respectively 

\begin{center}
$\{[u_{0,0},u_{0,4},u_{1,4},u_{2,4},u_{2,0},u_{1,0},u_{0,0}],[u_{0,1},u_{0,2},u_{1,2},u_{2,2},u_{2,1},u_{1,1},u_{0,1}],[u_{0,3},u_{1,3},u_{2,3},u_{0,3}]\}$, 

$\{[u_{0,0},u_{0,4},u_{1,4},u_{1,0},u_{0,0}],[u_{0,1},u_{0,2},u_{1,2},u_{1,1},u_{0,1}],[u_{0,3},u_{1,3},u_{2,3},u_{2,4},...,u_{2,2},u_{0,3}] \}$,

$\{[u_{0,0},u_{0,4},u_{1,4},u_{2,4},u_{2,0},u_{1,0},u_{0,0}],[u_{0,1},u_{0,2},u_{1,2},u_{1,1},u_{0,1}],[u_{0,3},u_{1,3},u_{2,3},u_{2,2},u_{2,1},u_{0,3}] \}$,

$\{[u_{0,0},u_{0,4},u_{1,4},u_{1,0},u_{0,0}], [u_{0,1},u_{0,2},u_{1,2},u_{2,2},u_{2,1},u_{1,1},u_{0,1}], [u_{0,3},u_{1,3},u_{2,3}, u_{2,4},u_{2,0},u_{0,3}] \}$ and 

$\{[u_{0,0},u_{0,4},u_{1,4},u_{1,0},u_{0,0}], [u_{0,1},u_{0,2},u_{1,2},u_{1,1},u_{0,1}],[u_{0,3},u_{1,3},u_{2,3}, u_{2,2},...,u_{2,4},u_{0,3}]  \}$.
\end{center}

The 2-factors for the second subcase, for $\ell = 0,...,4$ are respectively 

\begin{center}

$\{[u_{0,0},u_{0,4},u_{1,4},u_{2,4},u_{2,0},u_{1,0},u_{0,0}], [u_{0,1},u_{1,1},u_{2,1},u_{0,1}],[u_{0,2},u_{0,3},u_{1,3},u_{2,3},u_{2,2},u_{1,2},u_{0,2}] \}$,

$\{[u_{0,0},u_{0,4},u_{1,4},u_{1,0},u_{0,0}], [u_{0,1},u_{1,1},u_{1,2},u_{1,3},u_{2,3},u_{2,4},u_{2,0},u_{0,1}], [u_{0,2},u_{0,3},u_{2,2},u_{2,1},u_{0,2}] \}$,
$\{[u_{0,0},u_{0,4},u_{1,4},u_{1,3},u_{2,3},u_{0,0}], [u_{0,1},u_{1,1},u_{1,0},u_{2,0},u_{2,4},u_{0,1}], [u_{0,2},u_{0,3},u_{2,1},u_{2,2},u_{1,2},u_{0,2}] \}$,
$\{[u_{0,0},u_{0,4},u_{1,4},u_{2,4},u_{2,0},u_{1,0},u_{0,0}], [u_{0,1},u_{1,1},u_{2,1},u_{2,2},u_{2,3},u_{0,1}], [u_{0,2},u_{0,3},u_{1,3},u_{1,2},u_{0,2}] \}$ and

$\{[u_{0,0},u_{0,4},u_{1,4},u_{1,0},u_{0,0}], [u_{0,1},u_{1,1},u_{2,1},u_{2,0},...,u_{2,2},u_{0,1}], [u_{0,2},u_{0,3},u_{1,3},u_{1,2},u_{0,2}] \}$.
\end{center}
This completes the proof for $n=5$. We now look at $n=4$. Without loss of generality we need only to separate the vertices $u_{0,0}, u_{0,1}$ and $u_{0,2}$. The 2-factor which separates these for $\ell = 0,1,3$ are respectively 

\begin{center}

$\{[u_{0,0},u_{0,3},u_{1,3},u_{2,3},u_{2,0},u_{1,0},u_{0,0}], [u_{0,1},u_{1,1},u_{2,1},u_{0,1}], [u_{0,2},u_{1,2},u_{2,2},u_{0,2}] \}$, 

$\{[u_{0,0},u_{0,3},u_{1,3},u_{2,3},u_{0,0}],[u_{0,1},u_{1,1},u_{1,0},u_{2,0},u_{0,1}],[u_{0,2},u_{1,2},u_{2,2},u_{2,1},u_{0,2}] \}$ and  

$\{[u_{0,0},u_{0,3},u_{2,0},u_{1,0},u_{0,0}], [u_{0,1},u_{1,1},u_{2,1},u_{2,2},u_{0,1}],[u_{0,2},u_{1,2},u_{1,3},u_{2,3},u_{0,2}] \}$. 
\end{center}

It remains to show that the vertices cannot be separated when $\ell=2$. If they can be separated, then the separating 2-factor must consist of three 4-cycles because $C_3\Box_2 C_4$ has girth 4. The cycle $C$ containing 
$u_{0,1}$ must contain the 2-path $[u_{1,1},u_{0,1},u_{2,3}]$. This 2-path
does not belong to a 4-cycle because $u_{1,1}$ and $u_{2,3}$ have no common
neighbour.  Hence, there is no separating 2-factor.     

\end{proof}
\newpage
\begin{theorem}\label{T3.8}
The pseudo-Cartesian product $C_m \Box_\ell C_n$ is 3-spanning cyclable if and only if 
\begin{itemize}
\item $m\geq 4, n=3$ and $\ell=0$,
\item $m=3, n\geq5$,
\item $m=3,n=4$ and $\ell \neq 2$,
\item $m,n \geq 4$. 
\end{itemize}
\end{theorem}

\begin{proof}

Lemmas \ref{L3.6} and \ref{L3.7} cover the cases $n=3$ and $m=3$ respectively. We now look at $m\geq4$ and $n\geq 4$. Let $x,y$ and $z$ be three vertices to be separated by a 2-factor. We consider three cases depending on the number of columns in which $x,y$ and $z$ appear. \newline

Case 1: Three columns. At least one column does not contain one of the three vertices and so we may assume that $x$ lies in some column $i$, $y$ lies in column $j$, and $z$ in column $k$, where $0 \leq i < j < k \leq m-2$. Let the graph be $C_m \Box C_n$. The subgraph induced by columns $m-1, 0, 1, ..., i$ is isomorphic to $P_{i+2} \Box C_n$. The subgraph induced by columns $i+1, i+2, ..., j$ is isomorphic to $P_{j-i} \Box C_n$. Finally, the subgraph induced by columns $j+1, j+2, ...,m-2$ is isomorphic to $P_{m-2-j} \Box C_n$. Each of these three subgraphs is Hamiltonian by Lemma \ref{L3.3} and this yields a 2-factor separating $x,y$ and $z$. This 2-factor uses no edges between columns $m-2$ and $m-1$ and none of the three vertices lie in column $m-1$. Lemma \ref{L3.5} then implies there is a 2-factor separating any three vertices in distinct columns in $C_m \Box_\ell C_n$. \newline 

Case 2: Two columns. We may assume $x$ and $y$ both lie in column 0 and $z$ lies in some other column. The 2-factor can be formed as follows. We take the column containing $z$ as one cycle. The subgraph induced on the remaining vertices is isomorphic to $P_{m-1}\Box C_n$ no matter the value of $\ell$. It is 2-spanning cyclable by Theorem \ref{T3.4}. Thus there is a 2-factor of this subgraph separating $x,y$ and $z$. \newline

Case 3: One column. Without loss of generality we may assume that $x,y$ and $z$ lie in column 0 and that $x$ is the upper left vertex of the vertex array, that is, position $(0,n-1)$. There are some subcases to consider. 

We first assume that no two of $x,y,z$ are successive in the column. So let $z$ and $y$ be in the respective positions  $(0,i)$ and $(0,j)$, where $0 < i < j-1 < n-3$. Note that $n\geq6$ in this subcase.

\begin{figure}[H]
\centering

\begin{picture}(200,150)(-80,-25)

\multiput(-50,0)(20,0){4}{\multiput(0,0)(0,20){6}{\circle*{5}}}

\put(-60,20){z}
\put(-60,60){y}
\put(-60, 100){x}

\multiput(-50,40)(0,100){1}{\line(0,1){20}}
\multiput(-50,80)(0,100){1}{\line(0,1){20}}

\multiput(-30,40)(0,100){1}{\line(0,1){20}}

\multiput(-30,80)(0,100){1}{\line(0,1){20}}

\multiput(-10,20)(0,100){1}{\line(0,1){80}}

\multiput(10,20)(0,100){1}{\line(0,1){80}}

\multiput(-50,20)(0,100){1}{\line(1,0){40}}
\multiput(-50,0)(0,100){1}{\line(1,0){40}}

\multiput(-50,40)(0,100){1}{\line(1,0){20}}
\multiput(-50,60)(0,100){1}{\line(1,0){20}}
\multiput(-50,80)(0,100){1}{\line(1,0){20}}
\multiput(-50,100)(0,100){1}{\line(1,0){20}}

\qbezier(-50,0)(-20,15)(10,0)
\qbezier(-50,20)(-20,35)(10,20)
\qbezier(-10,0)(5,60)(-10,100)
\qbezier(10,0)(25,60)(10,100)

\multiput(50,0)(20,0){4}{\multiput(0,0)(0,20){7}{\circle*{5}}}

\put(40,40){z}
\put(40,80){y}
\put(40, 120){x}

\multiput(50,0)(0,100){1}{\line(0,1){20}}
\multiput(50,60)(0,100){1}{\line(0,1){20}}
\multiput(50,100)(0,100){1}{\line(0,1){20}}

\multiput(70,20)(0,100){1}{\line(0,1){20}}
\multiput(70,60)(0,100){1}{\line(0,1){20}}
\multiput(70,100)(0,100){1}{\line(0,1){20}}

\multiput(90,20)(0,100){1}{\line(0,1){100}}

\multiput(110,0)(0,100){1}{\line(0,1){20}}
\multiput(110,40)(0,100){1}{\line(0,1){80}}

\multiput(50,0)(0,100){1}{\line(1,0){40}}
\multiput(70,20)(0,100){1}{\line(1,0){20}}
\multiput(50,40)(0,100){1}{\line(1,0){20}}
\multiput(50,60)(0,100){1}{\line(1,0){20}}
\multiput(50,80)(0,100){1}{\line(1,0){20}}
\multiput(50,100)(0,100){1}{\line(1,0){20}}
\multiput(50,120)(0,100){1}{\line(1,0){20}}

\qbezier(50,20)(80,35)(110,20)
\qbezier(50,40)(80,55)(110,40)
\qbezier(90,0)(105,60)(90,120)
\qbezier(110,0)(125,60)(110,120)

\put(0,-35){{\sc Figure 3.} }
\end{picture}  \newline

\end{figure}

Consider Figure 3.  In both graphs we may subdivide the vertical edges between the top row
and the row below, and between the row containing $y$ and the row below it so that we may achieve an arbitrary gap between $x$ and $y$ and an arbitrary gap between $y$ and $z$. We may insert an arbitrary number of rows between the two bottom rows in the rightmost graph so that we have an arbitrary gap between $z$ and $x$ of two or more.  We may insert an
arbitrary number of columns between the third and fourth column without using any edges
between vertices of the two rightmost columns.  Thus, Lemma \ref{L3.5} takes care of this
case.

The next subcase is when $x$ and $y$ are successive and $z$ has a gap on both sides in the column. Note that this implies $n\geq5$. The case $n=5$ is special and we handle it separately. \newline 

\begin{figure}[H]
\centering

\begin{picture}(200,100)(-40,-20)

\multiput(-100,0)(20,0){4}{\multiput(0,0)(0,20){5}{\circle*{5}}}

\put(-115,20){z}
\put(-115,60){y}
\put(-115, 80){x}

\put(40,20){z}
\put(40,60){y}
\put(40, 80){x}

\put(190,20){z}
\put(190,60){y}
\put(190, 80){x}

\multiput(-100,40)(0,100){1}{\line(0,1){20}}
\multiput(-80,40)(0,100){1}{\line(0,1){20}}
\multiput(-60,0)(0,100){1}{\line(0,1){20}}
\multiput(-60,40)(0,100){1}{\line(0,1){40}}
\multiput(-40,0)(0,100){1}{\line(0,1){20}}
\multiput(-40,40)(0,100){1}{\line(0,1){40}}

\multiput(-100,0)(0,100){1}{\line(1,0){20}}
\multiput(-100,20)(0,100){1}{\line(1,0){40}}
\multiput(-100,40)(0,100){1}{\line(1,0){20}}
\multiput(-60,40)(0,100){1}{\line(1,0){20}}
\multiput(-100,60)(0,100){1}{\line(1,0){20}}
\multiput(-100,80)(0,100){1}{\line(1,0){20}}

\qbezier(-100,20)(-70,35)(-40,20)
\qbezier(-100,0)(-115,40)(-100,80)
\qbezier(-80,0)(-95,40)(-80,80)
\qbezier(-60,0)(-45,40)(-60,80)
\qbezier(-40,0)(-25,40)(-40,80)

\multiput(50,0)(20,0){4}{\multiput(0,0)(0,20){5}{\circle*{5}}}

\multiput(50,0)(0,100){1}{\line(0,1){20}}
\multiput(50,40)(0,100){1}{\line(0,1){20}}

\multiput(70,0)(0,100){1}{\line(0,1){20}}
\multiput(70,40)(0,100){1}{\line(0,1){20}}

\multiput(90,40)(0,100){1}{\line(0,1){40}}
\multiput(90,0)(0,100){1}{\line(0,1){20}}
\multiput(110,20)(0,100){1}{\line(0,1){20}}
\multiput(110,60)(0,100){1}{\line(0,1){20}}

\multiput(90,0)(0,100){1}{\line(1,0){20}}
\multiput(90,20)(0,100){1}{\line(1,0){20}}
\multiput(90,40)(0,100){1}{\line(1,0){20}}

\multiput(50,0)(0,100){1}{\line(1,0){20}}

\multiput(50,20)(0,100){1}{\line(1,0){20}}
\multiput(50,40)(0,100){1}{\line(1,0){20}}
\multiput(50,60)(0,100){1}{\line(1,0){20}}
\multiput(50,80)(0,100){1}{\line(1,0){40}}

\qbezier(50,80)(80,70)(110,60)
\qbezier(110,0)(125,40)(110,80)

\multiput(200,0)(20,0){4}{\multiput(0,0)(0,20){5}{\circle*{5}}}

\multiput(200,0)(0,100){1}{\line(0,1){20}}
\multiput(200,40)(0,100){1}{\line(0,1){20}}

\multiput(220,0)(0,100){1}{\line(0,1){20}}
\multiput(220,40)(0,100){1}{\line(0,1){20}}

\multiput(240,0)(0,100){1}{\line(0,1){60}}

\multiput(260,0)(0,100){1}{\line(0,1){40}}
\multiput(260,60)(0,100){1}{\line(0,1){20}}

\multiput(200,0)(0,100){1}{\line(1,0){20}}
\multiput(200,20)(0,100){1}{\line(1,0){20}}
\multiput(200,40)(0,100){1}{\line(1,0){20}}
\multiput(200,60)(0,100){1}{\line(1,0){20}}
\multiput(240,60)(0,100){1}{\line(1,0){20}}
\multiput(200,80)(0,100){1}{\line(1,0){40}}

\qbezier(200,80)(230,60)(260,40)
\qbezier(240,0)(225,40)(240,80)
\qbezier(260,0)(275,40)(260,80)

\put(50,-35){{\sc Figure 4.} }
\end{picture}  \newline

\end{figure}

Figure 4 shows 2-factors separating $x,y$ and $z$ for $C_4 \Box_\ell C_5$ for $\ell = 0,1$ and 2. We can flip the right and middle graph to obtain covers for the cases $\ell=3$ and 4 respectively. It is easy to see that we can add any number of columns to each of the initial figures so that $C_m \Box_\ell C_5$ has a 2-factor separating $x,y$ and $z$ for all $m\geq4$. 

\begin{figure}[H]
\centering

\begin{picture}(200,160)(-80,-30)

\multiput(0,0)(20,0){4}{\multiput(0,0)(0,20){6}{\circle*{5}}}

\put(-15,40){z}
\put(-15,80){y}
\put(-15, 100){x}

\multiput(0,60)(0,100){1}{\line(0,1){20}}

\multiput(20,60)(0,100){1}{\line(0,1){20}}

\multiput(40,0)(0,100){1}{\line(0,1){20}}
\multiput(40,40)(0,100){1}{\line(0,1){60}}

\multiput(60,0)(0,100){1}{\line(0,1){20}}
\multiput(60,40)(0,100){1}{\line(0,1){60}}

\multiput(0,0)(0,100){1}{\line(1,0){20}}
\multiput(0,20)(0,100){1}{\line(1,0){40}}
\multiput(0,40)(0,100){1}{\line(1,0){40}}
\multiput(0,60)(0,100){1}{\line(1,0){20}}
\multiput(0,80)(0,100){1}{\line(1,0){20}}
\multiput(0,100)(0,100){1}{\line(1,0){20}}

\qbezier(0,0)(-15,60)(0,100)
\qbezier(20,0)(4,60)(20,100)
\qbezier(40,0)(55,60)(40,100)
\qbezier(60,0)(75,60)(60,100)
\qbezier(0,20)(30,5)(60,20)
\qbezier(0,40)(30,55)(60,40)

\put(0,-35){{\sc Figure 5.} }
\end{picture}  \newline

\end{figure}

Now let $n\geq 6$. This means there is a gap with at least two vertices. Figure 5 shows a typical situation but note that $z$ also could be in row 1 instead of row 2 as shown. We may subdivide the vertical edges between the row containing $y$ and the row below to obtain an arbitrary gap between $y$ and $z$. Similarly, we may add an arbitrary number of rows between the two bottom rows to obtain an arbitrary gap between $z$ and $x$. It is easy to see that we can add any number of columns between the last two columns without using any edges between the last two columns. This completes the case that two of the vertices are successive.  
\newline 

\begin{figure}[H]
\centering

\begin{picture}(200,90)(-30,-30)

\multiput(-100,0)(20,0){4}{\multiput(0,0)(0,20){4}{\circle*{5}}}

\put(-110,20){z}
\put(-110,40){y}
\put(-110, 60){x}

\put(40,20){z}
\put(40,40){y}
\put(40, 60){x}

\put(190,20){z}
\put(190,40){y}
\put(190, 60){x}

\multiput(-100,0)(0,100){1}{\line(0,1){20}}
\multiput(-40,0)(0,100){1}{\line(0,1){20}}

\multiput(-100,0)(0,100){1}{\line(1,0){60}}
\multiput(-100,20)(0,100){1}{\line(1,0){60}}
\multiput(-100,40)(0,100){1}{\line(1,0){60}}
\multiput(-100,60)(0,100){1}{\line(1,0){60}}

\qbezier(-100,40)(-65,55)(-40,40)
\qbezier(-100,60)(-65,75)(-40,60)

\multiput(50,0)(20,0){4}{\multiput(0,0)(0,20){4}{\circle*{5}}}

\multiput(50,0)(0,100){1}{\line(0,1){20}}
\multiput(70,0)(0,100){1}{\line(0,1){20}}
\multiput(90,0)(0,100){1}{\line(0,1){40}}
\multiput(110,0)(0,100){1}{\line(0,1){20}}
\multiput(110,40)(0,100){1}{\line(0,1){20}}

\multiput(50,0)(0,100){1}{\line(1,0){20}}
\multiput(90,0)(0,100){1}{\line(1,0){20}}
\multiput(50,20)(0,100){1}{\line(1,0){20}}
\multiput(50,40)(0,100){1}{\line(1,0){40}}
\multiput(50,60)(0,100){1}{\line(1,0){60}}

\qbezier(50,40)(80,30)(110,20)
\qbezier(50,60)(80,50)(110,40)

\multiput(200,0)(20,0){4}{\multiput(0,0)(0,20){4}{\circle*{5}}}

\multiput(200,0)(0,100){1}{\line(0,1){20}}
\multiput(220,0)(0,100){1}{\line(0,1){20}}
\multiput(240,0)(0,100){1}{\line(0,1){40}}
\multiput(260,20)(0,100){1}{\line(0,1){40}}

\multiput(200,0)(0,100){1}{\line(1,0){20}}
\multiput(240,0)(0,100){1}{\line(1,0){20}}
\multiput(200,20)(0,100){1}{\line(1,0){20}}
\multiput(200,40)(0,100){1}{\line(1,0){40}}
\multiput(200,60)(0,100){1}{\line(1,0){60}}

\qbezier(200,40)(230,20)(260,0)
\qbezier(200,60)(230,40)(260,20)

\put(50,-35){{\sc Figure 6.} }
\end{picture}  \newline

\end{figure}

We now consider the case that the three vertices are successive. Figure 6 provides a solution for $C_4 \Box_\ell C_4$ for all $\ell$. 
We can easily add any number of columns giving us a solution for $C_m \Box_\ell C_4$ for all $m \geq 4$. \newline

\begin{figure}[H]
\centering

\begin{picture}(200,100)(-30,-20)

\multiput(-100,0)(20,0){4}{\multiput(0,0)(0,20){5}{\circle*{5}}}

\put(-115,40){z}
\put(-115,60){y}
\put(-115, 80){x}

\put(35,40){z}
\put(35,60){y}
\put(35, 80){x}

\put(190,40){z}
\put(190,60){y}
\put(190, 80){x}

\multiput(-100,20)(0,100){1}{\line(0,1){20}}
\multiput(-80,20)(0,100){1}{\line(0,1){20}}
\multiput(-60,0)(0,100){1}{\line(0,1){60}}
\multiput(-40,0)(0,100){1}{\line(0,1){60}}

\multiput(-100,0)(0,100){1}{\line(1,0){20}}
\multiput(-100,20)(0,100){1}{\line(1,0){20}}
\multiput(-100,40)(0,100){1}{\line(1,0){20}}
\multiput(-100,60)(0,100){1}{\line(1,0){40}}
\multiput(-100,80)(0,100){1}{\line(1,0){20}}
\multiput(-60,80)(0,100){1}{\line(1,0){20}}

\qbezier(-100,0)(-115,40)(-100,80)
\qbezier(-80,0)(-95,40)(-80,80)
\qbezier(-60,0)(-75,40)(-60,80)
\qbezier(-40,0)(-25,40)(-40,80)
\qbezier(-100,60)(-70,75)(-40,60)

\multiput(50,0)(20,0){4}{\multiput(0,0)(0,20){5}{\circle*{5}}}

\multiput(50,20)(0,100){1}{\line(0,1){20}}
\multiput(90,20)(0,100){1}{\line(0,1){20}}
\multiput(110,0)(0,100){1}{\line(0,1){40}}
\multiput(110,60)(0,100){1}{\line(0,1){20}}

\multiput(50,0)(0,100){1}{\line(1,0){40}}
\multiput(50,20)(0,100){1}{\line(1,0){40}}
\multiput(50,40)(0,100){1}{\line(1,0){40}}
\multiput(50,60)(0,100){1}{\line(1,0){60}}
\multiput(50,80)(0,100){1}{\line(1,0){40}}

\qbezier(50,0)(35,40)(50,80)
\qbezier(90,0)(75,40)(90,80)
\qbezier(110,0)(125,40)(110,80)
\qbezier(50,60)(80,50)(110,40)

\multiput(200,0)(20,0){4}{\multiput(0,0)(0,20){5}{\circle*{5}}}

\multiput(200,0)(0,100){1}{\line(0,1){40}}
\multiput(220,20)(0,100){1}{\line(0,1){20}}
\multiput(240,40)(0,100){1}{\line(0,1){20}}
\multiput(260,40)(0,100){1}{\line(0,1){40}}

\multiput(200,00)(0,100){1}{\line(1,0){60}}
\multiput(220,20)(0,100){1}{\line(1,0){40}}
\multiput(220,40)(0,100){1}{\line(1,0){20}}
\multiput(200,60)(0,100){1}{\line(1,0){40}}
\multiput(200,80)(0,100){1}{\line(1,0){60}}

\qbezier(200,40)(230,20)(260,0)
\qbezier(200,60)(230,40)(260,20)
\qbezier(200,80)(230,60)(260,40)

\put(50,-35){{\sc Figure 7.} }
\end{picture}  \newline

\end{figure}

Figure 7 produces solutions for three successive vertices and $n=5$. An arbitrary number of columns may be added to any of the graphs so that we have solutions for $C_m \Box_\ell C_5$ for all $m\geq4$. 

It remains to look at the subcase when $m\geq4$ and $n\geq 6$. \newline

\begin{figure}[H]
\centering

\begin{picture}(200,140)(-80,-30)

\multiput(0,0)(20,0){4}{\multiput(0,0)(0,20){6}{\circle*{5}}}

\put(-15,60){z}
\put(-15,80){y}
\put(-15, 100){x}

\multiput(0,40)(0,100){1}{\line(0,1){20}}
\multiput(20,40)(0,100){1}{\line(0,1){40}}
\multiput(40,0)(0,100){1}{\line(0,1){20}}
\multiput(40,40)(0,100){1}{\line(0,1){60}}
\multiput(60,0)(0,100){1}{\line(0,1){20}}
\multiput(60,40)(0,100){1}{\line(0,1){20}}
\multiput(60,80)(0,100){1}{\line(0,1){20}}

\multiput(0,0)(0,100){1}{\line(1,0){20}}
\multiput(0,20)(0,100){1}{\line(1,0){40}}
\multiput(20,40)(0,100){1}{\line(1,0){20}}
\multiput(0,80)(0,100){1}{\line(1,0){20}}
\multiput(0,100)(0,100){1}{\line(1,0){20}}

\qbezier(0,0)(-15,50)(0,100)
\qbezier(20,0)(5,50)(20,100)
\qbezier(40,0)(55,50)(40,100)
\qbezier(60,0)(75,50)(60,100)

\qbezier(0,20)(30,35)(60,20)
\qbezier(0,40)(30,55)(60,40)
\qbezier(0,60)(30,75)(60,60)
\qbezier(0,80)(30,95)(60,80)

\put(0,-35){{\sc Figure 8.} }
\end{picture}  \newline

\end{figure}

Consider Figure 8 above. It is easy to see that we can add any number of rows between the bottom two rows. We can also add any number of columns between the last two columns. Since there are no edges between the last two columns and using Lemma \ref{L3.5} we have that $C_m \Box_\ell C_n$ separates $x,y$ and $z$ for all $m\geq 4$ and all $n\geq6$. This takes care of all cases and proves Theorem \ref{T3.8}. 

\end{proof}

\section{The special case}
In this section we examine the special case that the graph is isomorphic to $Y\Box K_2$, where $Y$ is the circulant graph of even order $n$ with connection set $\{\pm 1,n/2\}$.
For convenience and a more general understanding we define a generalisation of the psuedo-Cartesian product of two cycles by  
$C_m \Box_\tau C_n$, where $\tau \in S_n$ denotes a permutation. The graph is obtained by starting with the Cartesian product $P_m \Box C_n$ and adding edges from $u_{m-1,j}$ to $u_{0,\tau(j)}$.

It is easily verifiable that the special case graph is isomorphic to $C_m \Box_{\tau} C_4$, where $m\geq3$ and $\tau = (0\;3)(1\;2)$. Note that a convenient automorphism for this graph is $\rho$ which maps $u_{i,j}$ to $u_{i+1,j}$ for $i<m-1$, and $u_{m-1,j}$ to $u_{0,\tau(j)}$. Another convenient automorphism for this special case graph is $\alpha(u_{i,j}) = u_{i,3-j}$. 

\begin{theorem}\label{T4.1}
The graph $C_m \Box_{\tau} C_4$, where $\tau = (0\;3)(1\;2)$, is both 2-spanning and 3-spanning cyclable for $m\geq3$. 
\end{theorem}

\begin{proof}
Theorem \ref{T3.4} shows that the graph is 2-spanning cyclable for $m\geq 3$. For the 3-spanning case we have three cases. The first is when all three vertices lie in different columns and we leave the easy verification that there is a 2-factor separating the three vertices to the reader. 

The second case is when two vertices $x,y$ are in the same column and the other vertex $z$ is not. When $m=3$ it is easy to verify. Suppose that $m \geq4$. In this case we can use the automorphism $\rho$ repeatedly until $z$ is mapped to column 0. We can then use the column 0 cycle to contain that vertex and use Theorem \ref{T3.4} to separate the remaining vertices. 

This then leaves us with the third case in which all three vertices are in the same column. This of course means that all three vertices must be in different rows. Using the automorphism $\rho$ we can assume that all three vertices are located in column 0. Also, using the automorphism $\alpha$ we see that we only need to separate the  triplets $\{ u_{0,0},u_{0,1},u_{0,2}\} $ and $ \{ u_{0,0},u_{0,1},u_{0,3}\}$. It is easy to verify this for $m=3$ and it is left to prove it for $m\geq4$. 

We only need one 2-factor which consists of the cycles 
\begin{center}
$[u_{0,0},u_{1,0},u_{1,3},...,u_{m-1,3},u_{0,0}]$,
\end{center}

\begin{center}
$[u_{0,1},u_{1,1},u_{1,2},...,u_{m-1,2},u_{0,1} ]$, and 
\end{center}

\begin{center}
$[u_{0,3},u_{m-1,0},...,u_{2,0},u_{2,1},...,u_{m-1,1},u_{0,2},u_{0,3} ]$. 
\end{center}
 This 2-factor separates the triplets $\{ u_{0,0},u_{0,1},u_{0,2}\}$ and  $\{u_{0,0},u_{0,1},u_{0,3}\}$.  
\end{proof}

\section{The circulant case}
We will assume that the vertices on the graph  $\mathrm{circ(n; \pm 1, \pm s)}$ are cyclically labelled $u_i$, where $0 \leq i \leq n-1$, and where the index $i$ is computed modulo $n$. We also assume that $n \geq 2s$. We define the automorphism $\pi$ by $\pi(u_i) = u_{i+1}$. 

Let $P[x,y]$ denote the path 
$[u_x,u_{x+1},u_{x+2},\ldots,u_{y-1},u_y]$. Note that $P[x,x+1]$ is the edge $[u_x,u_{x+1}]$,
whereas the path $P[x+1,x]$ is a Hamilton path whose terminal vertices are $u_x$ and $u_{x+1}$. For $x-y$ even, we denote the path $[u_x, u_{x+2}, u_{x+4}, \ldots, y]$ by $P^2[x,y]$.  An important convention is that $P[x,x]$ and $P^2[x,x]$ both denote the single vertex $u_x$.

\begin{theorem}\label{T5.1} The graph $\mathrm{circ(n; \pm 1, \pm s)}$, where $s \geq 2$, $gcd(n,s)=1$ and $n \geq 2s$, is 2-spanning cyclable if and only if $n\geq 6$. 
\end{theorem}

\begin{proof}

Since there must be at least 2 cycles both of length at least 3 we have that $n\geq6$. 
 
 We first consider $s=2$. Note that in this case $n$ must be odd. We show that we can always separate $u_0$ from any vertex $u_i$ where $i\neq 0$. In this case we consider the 2-factor consisting of the cycles $[u_{n-1}, u_{n-2}, u_{n-3}, u_{n-1}]$ and $P^2[0,n-5] \cup P[n-5,n-4] \cup P^{2}[1,n-4] \cup P[0,1]$. Using this 2-factor and the automorphism $\pi$ we can separate $u_0$ from all $u_i$, where $i\neq0$. 

We now consider $s\geq 3$. We start by constructing a 2-factor $F$ consisting of the cycles $[u_0,u_s] \cup P[s,n-s] \cup [u_{n-s},u_0]$ and $P[1,s-1] \cup [u_{s-1},u_{n-1}] \cup P[n-s+1,n-1] \cup [u_{n-s+1},u_1]$. This 2-factor will separate $u_0$ from the vertices in the second cycle. To separate $u_0$ from all other vertices we consider $\pi(F)$ and $\pi^{-1}(F)$. 
\end{proof}

We now move onto the 3-spanning cyclability of circulant graphs. 

\begin{theorem}\label{T5.2}
The graph $\mathrm{circ(n; \pm 1, \pm 2)}$ is not 3-spanning cyclable. 
\end{theorem}
\begin{proof}
Consider separating the vertices $u_0, u_1$ and $u_2$. The vertex $u_1$ is forced to be  adjacent to $u_3$ in its cycle. The vertex $u_2$ cannot be adjacent to any of $u_0,u_1$ and $u_3$ and hence we have a contradiction. 
\end{proof}

\begin{theorem}\label{5.3}
The graph $\mathrm{circ(n; \pm 1, \pm s)}$, where $n=2s+1$ or $n=2s+2$, is not 3-spanning cyclable. 
\end{theorem}

\begin{proof}
We first look at the case $n=2s+1$. Consider separating the vertices $u_0, u_s$ and $u_{n-s}$. The vertex $u_0$ is forced to be adjacent to $u_1$ in its cycle. The vertex $u_{n-s}$ is also forced to be adjacent to $u_1$ in its cycle and we reach a contradiction. 

We now look at the case $n=2s+2$. 
Consider separating the vertices $u_0, u_1$ and $u_{s+1}$. The vertex $u_1$ can only be adjacent to $u_{2}$ and $u_{n -s +1}$ in its cycle. 

Suppose that vertex $u_{s+1}$ is adjacent to $u_{n-s}$ in its cycle. From that vertex it cannot move to vertex $u_{n -s +1}, u_0$ or $u_2$ and hence we reach a contradiction. So $u_{s+1}$ must only be adjacent to $u_{n-1}$ and $u_{s}$ in its cycle. 

Finally, $u_0$ can only be adjacent to $u_{n-s}$ and we reach a contradiction. 
\end{proof}

We now show that the graph $\mathrm{circ(n; \pm 1, \pm s)}$, where $s\geq3$ is 3-spanning cyclable for sufficiently large $n$. 

\begin{theorem}\label{T5.4}
The graph $\mathrm{circ(n; \pm 1, \pm s)}$, where $s\geq3$, is 3-spanning cyclable for $n\geq 4s+3$. 
\end{theorem}

\begin{proof}
We assume $s\geq 3$ because of Theorem \ref{T5.2}. Let $C[x]$ be the 4-cycle 
$[u_x,u_{x+1},u_{x+s+1},u_{x+s},u_x]$. 
We now show that any three distinct vertices can be separated when $n\geq 4s+3$. Without loss
of generality suppose that one of the three vertices is $u_0$. Let the other two vertices be 
$u_i$ and $u_j$, where $n-1\geq j>i\geq 1$. Because of the automorphism interchanging the
vertices $u_k$ and $u_{n-k}$, $0\leq k\leq n/2$, we may assume that $n-j\geq i$. Further, we
may assume $j-i\geq i$ because we can cyclically relabel the vertices by subtracting $i$ from
each subscript. There are three cases to consider.

The strategy in all the cases is the same.  We choose two 4-cycles each of which contains one
of the target vertices and then we find a third cycle containing the remaining target vertex 
and the rest of the vertices, thereby yielding a 2-factor separating the three vertices.  
We call this {\it 2-factor completion} and provide details in the first case below.  

The first case is $i>s$ and we break this into two subcases.  When $i=s+1$ or $i=s+2$,
use the 4-cycle $C[1]$ containing $u_i$ and the 4-cycle $C[j-1]$ containing $u_j$. To form
the cycle containing $u_0$ which completes a 2-factor, use the paths $P[3,s]$, $P[s+3,j-2]$, 
$P[j+1,j-2+s]$ and $P[j+s+1,0]$ joined by the edges $[u_0,u_s],[u_3,u_{s+3}],[u_{j-2},
u_{j-2+s}]\mbox{ and }[u_{j+1},u_{j+s+1}]$.  It is straightforward to verify that the 
vertices required for the latter cycle are available.
When $i>s+2$, use the 4-cycle $C[n-s]$ containg $u_0$ and the 4-cycle $C[i-s]$ containing
$u_i$.  Complete to a 2-factor as in the preceding subcase.

The second case is $i=s$ and we also break this into three subcases.  When $j\in\{2s,2s+1\}$,
we use the 4-cycle $C[2s]$ containing $u_j$ and the 4-cycle $C[n-s]$ containing $u_0$.  The
2-factor completion is straightforward.  When $3s+2\geq j\geq 2s+2$, we use the 4-cycle 
$C[s-1]$ containing $u_i$ and the 4-cycle $C[j]$ containing $u_j$.  The 2-factor completion
again is straightforward. Finally, when $j\geq 3s+3$, we use the 4-cycle $C[s-1]$ as before 
and the 4-cycle $C[j-s]$ containing $u_j$.  Perform the 2-factor completion as before.

The last case is $i<s$ and there are four subcases. If $j\geq 3s+3$, then use the 4-cycle
$C[j-s-1]$ containing $u_j$ and the 4-cycle $C[i]$ containing $u_i$.  Carry out the usual
2-factor completion to finish this subcase.  If $2s+2\leq j\leq 3s+2$, then use the 4-cycle
$C[i]$ containing $u_i$. If $j = 3s+2$ we use the 4-cycle $C[j-1]$ to contain $u_j$ otherwise we use $C[j]$. The 2-factor completion
is straightforward.  Note that this subcase requires that $n\geq 4s+3$.

When $j\in\{2s,2s+1\}$, use the 4-cycle $C[2s]$ containing $u_j$.  There are two subcases.
When $i=1$, use the 4-cycle $C[1]$ when $s>3$, and $C[n-s+1]$ when $s=3$, to contain $u_i$, but when $i>1$, use the 4-cycle $C[n-s]$
containing $u_0$. In both situations it is easy to carry out a 2-factor completion.

The final subcase is $i<j<2s$.  In this subcase we use the 4-cycle $C[j]$ containing $u_j$
and the 4-cycle $C[n-s-1]$ containing $u_0$.  The completion to a 2-factor is straightforward.
This completes the proof.
\end{proof}


\begin{thebibliography}{9999}

\bibitem{A1} B. Alspach, A. Joshi, On The 2-Spanning Cyclability Of Honeycomb
Toroidal Graphs, (arXiv:2309.04938 [math.CO]). 

\bibitem{L1} C.-K. Lin, J.-M. Tan, L.-H. Hsu and T.-L. Kung, Disjoint cycles in hypercubes with prescribed vertices in each cycle, {\sl Discrete Appl. Math.} {\bf 161} (2013), 2992--3004.

\bibitem{Q1} H. Qiao, E. Sabir and J. Meng, The spanning cyclability of Cayley graphs generated by transposition trees, {\sl Discrete Appl. Math.} {\bf 328} (2023), 60--69.


\bibitem{Y1} M.-C. Yang, L.-H. Hsu, C.-N. Hung and E. Cheng, 2-Spanning cyclability problems of some generalized Petersen graphs,  {\sl Discuss. Math. Graph Theory} {\bf 40} (2020), 713--731.


\end{thebibliography}
\end{document}